\documentclass[11pt,article]{amsart}
\usepackage{amssymb}
\usepackage{amsfonts}
\usepackage{amsbsy}
\usepackage{latexsym}
\usepackage{graphicx}
\usepackage{amsmath}
\usepackage[all]{xy}

\newtheorem{theorem}{Theorem}[section]

\newtheorem{proposition}[theorem]{Proposition}
\newtheorem{corollary}[theorem]{Corollary}

\theoremstyle{definition}
\newtheorem{definition}[theorem]{Definition}
\newtheorem{example}[theorem]{Example}

\theoremstyle{remark}
\newtheorem{remark}[theorem]{Remark}

\numberwithin{equation}{section}

\newcommand\M{\mathbb{M}}
\newcommand\cA{\mathcal{A}}

\newcommand\cT{\mathcal{T}}

\def\sideremark#1{\ifvmode\leavevmode\fi\vadjust{\vbox
to0pt{\vss \hbox to 0pt{\hskip\hsize\hskip1em
\vbox{\hsize2cm\tiny\raggedright\pretolerance10000
\noindent#1\hfill}\hss}\vbox to8pt{\vfil}\vss}}}

\begin{document}

\title{Some Structural Properties of homomorphism dilation systems  for Linear Maps}

\author{Deguang Han}
\address{Department of Mathematics, University of Central
Florida, Orlando, USA} \email{deguang.han@ucf.edu}

\author{David R. Larson}
\address{Department of Mathematics, Texas A\&M University, College Station, USA}
\email{larson@math.tamu.edu}

\author{Bei Liu}
\address{Department of Mathematics, Tianjin University of Technology, Tianjin, China}
\email{beiliu1101@gmail.com}

\author{Rui Liu}
\address{Department of Mathematics and LPMC, Nankai University, Tianjin, China}
\address{Department of Mathematics, Texas A\&M University, College Station, USA}
\email{ruiliu@nankai.edu.cn; rliu@math.tamu.edu}

\begin{abstract} Inspired by some recent development on the theory about projection valued dilations for operator valued measures or more generally  bounded homomorphism dilations for  bounded linear maps on Banach algebras,  we explore a pure algebraic version of the dilation theory for linear systems acting on unital algebras and vector spaces. By introducing two natural dilation structures, namely the canonical and the universal dilation systems, we prove that every linearly minimal dilation is equivalent to a reduced homomorphism dilation of the universal dilation, and all the  linearly minimal homomorphism dilations can be classified by  the associated reduced subspaces contained in the kernel of synthesis operator for the universal dilation.
\end{abstract}

\thanks{Deguang Han acknowledges partial support by NSF grants DMS-1106934 and DMS-1403400. Bei Liu and Rui Liu  both are supported by NSFC grants 11201336 and 11001134}

\date{}



\maketitle

\section{Introduction} In a recent AMS Memoir \cite{HLLL-AMS} we
established a general dilation theory for operator
valued measures acting on Banach spaces where the operator-valued
measures are not necessarily completely bounded.  This  naturally extends to bounded linear maps acting on Banach algebras and Banach spaces, which can be viewed as a noncommutative analogue for the dilations of operator valued measures. This investigation was mainly motivated by some recent dilation results in frame theory (c.f. \cite{CHL, GH, GH2, HL, Han}), in particular, by a general dilation theorem for framings established by Casazza, Han and Larson \cite{CHL} which states that  every framing  (even for a Hilbert space) can have a basis dilation which is highly `` non-Hilbertian" in nature and the dilation space has to be a Banach space in general. This  is viewed as a
true generalization of the well-known Naimark dilation theory \cite{Naimark-1, Naimark-2, Pa} for positive operator valued measures, in which case the Hilbertian structure can be completely captured by the dilation space. The Naimark dilation theorem states that every positive operator valued measure has a (self-adjoint) projection valued dilation acting on a Hilbert space. We built in  \cite{HLLL-AMS, HLLL-F, HLLL-Con}  some interesting connections between frame
theory  and dilations of operator-valued measures on one hand,   and the dilations of bounded linear maps between von Neumann algebras on the other hand.  It was proved that  any operator-valued measure, not necessarily completely bounded, always has a dilation to a projection (idempotent) valued measure acting on a Banach space. More generally,  every bounded linear map acting on a Banach algebra  has a bounded homomorphism dilation acting on a Banach space.   Here the bounded linear map needs not to be completely bounded and the
dilation space often needs to be a Banach space even if the underlying space is a Hilbert space, and the underlying algebra is a von Neumann algebra. A typical example is the transpose map on the algebra $B(H)$ of all bounded linear operators on a Hilbert space $H$. This map is not completely bounded but it has a bounded homomorphism dilation
on a Banach space and the dilation space can never be taken as a Hilbert space. Such examples also exist for commutative $C^{*}$-algebras \cite{HLLL-AMS}. Therefore the bounded homomorphism dilation theory for any bounded linear maps truly generalizes the Stinespring's dilation theorem (c.f. \cite{Aev, HA, Pa, St}) which states that a bounded linear map on a $C^{*}$-algebra  admits a $*$-homomorphism dilation (acting on a Hilbert space) if and only if it is completely bounded. It was pointed out in \cite{HLLL-AMS} that the problem for the existence of a non-completely bounded linear map that admits a Hilbertian bounded homomorphism dilation is equivalent to Kadison's similarity problem \cite{Ka}.   All these results indicate that it  might be possible to establish some kind of classification theory for  bounded linear maps  based on the properties of their dilations for more general Banach algebras and Banach spaces. 

 In the dilation theorems for general operator valued measures or general bounded linear maps the dilation Banach space was built on a natural ``smallest" dilation vector space equipped with a proper dilation norm  so that the involved homomorphisms and linear maps are continuous with respect to the dilation norm.  However, neither the (algebraic)  dilation space  nor the dilation norm is in general not unique. So it seems that there might be  some structural theory involved in the ``classification" of bounded linear map based on the dilations spaces and the dilation norms, and the completely bounded maps belong to a special class within this structural theory. In order to understand the topological nature of the dilation theory for continuous maps,  a good understanding on the purely algebraic aspects of the dilation theory for linear maps is naturally needed. However it seems to us that  there is no systematic investigation  (at least we are not aware of ) in the literature so far.  Our aim of this paper is to present several structural results involving the classification of algebraic homomorphism dilations for linear maps acting on general vector spaces. With our ultimate goal of establishing a classification theory of Banach space homomorphism dilations on various dilations spaces, we hope that this paper serves as a first step of this effort.

The rest of the paper is organized as follows: In section 2 we introduce two natural dilations, the canonical dilation and the universal dilation. While the canonical dilation serves as the ``smallest" dilation system,  the universal one indeed serves as the ``largest" dilation  system. Naturally we prove that all the irreducible dilations are equivalent to the canonical dilation, and every dilation is equivalent to a reduced dilation of the universal dilation.  The main classification results are presented in section 3 in which all the dilations are classified by their  associated reduced subspaces contained in the kernel of synthesis operator from the universal dilation. We provide a few remarks and examples in section 4 to demonstrate the complexity and the rich structure of the algebraic dilation theory. 

\section{Principle and Universal Dilations}

A {\it linear system}  is a triple $(\varphi, \mathcal{A}, V)$ such that $\varphi$ is a unital linear map from  a unital algebra  $\mathcal{A}$ to $L(V)$, where $V$ is a vector space and $L(V)$ denotes the space of all linear maps from $V$ to $V$. In the case that $\mathcal{A}$ is well understood in the discussion we will  usually skip $\mathcal{A}$ from the notation.

\begin{definition}  A homomorphism dilation system   of  a linear system $(\varphi, V)$ is a unital homomorphism $\pi$ from $\mathcal{A}$ to a linear operator space $L(W)$ for some vector space $W$ such that there exist an injective linear map $T: V\rightarrow W$ and a surjective linear map $S:W\rightarrow V$ such that for all $a\in \cA$ the following diagram commutes
\[   \xymatrix{ W\ar[rr]^{\pi(a)} & & W\ar[d]^{S} \\
V\ar[u]^{T}\ar[rr]^{\varphi(a)} & & V  }\]
That is,
$$
\varphi(a) = S\pi(a)T, \  \  \  \forall a\in \mathcal{A}.
$$

\end{definition}

We will use $(\pi,  S, T, W)$ to denote this homomorphism dilation system,  and the dimension of $W$ is called the {\it dilation dimension} of the homomorphism dilation system $(\pi, S, T, W)$. For our convenience we call $T$ as the {\it analysis operator} and $S$ as the {\it synthesis operator} for the dilation system.  If $ker(S) $ contains a nonzero $\pi$-invariant subspace, then we say that $(\pi, S, T, W)$ is {\it reducible}, and otherwise it is called {\it irreducible}.

Suppose that $K$ is a $\pi$-invariant nonzero subspace of $ker (S)$. Define $\tilde{W} = W/K$, and let $\tilde{S}: \tilde{W} \rightarrow V$, $\tilde{T}: V\rightarrow \tilde{W}$ and $\tilde{\pi}: \mathcal{A} \rightarrow L(\tilde{W})$ be the induced linear maps. Then we have  for any $a\in\mathcal{A}$ and any $v\in V$ that
$$
\tilde{S}\tilde{\pi}\tilde{T}(v) = \varphi(a)v.
$$
Thus $(\tilde{\pi}, \tilde{S}, \tilde{T}, \tilde{W}) $ is an homomorphism dilation of $(\varphi, V)$  and we call it a {\it  reduced homomorphism dilation}  of $(\pi, S, T, W)$ associated with $K$. If $K$ is the maximal $\pi$-invariant subspace contained in $ker(S)$, then it is easy to show that $ker(\tilde{S})$ does not contain any nonzero $\tilde{\pi}$-invariant subspace anymore, and hence the reduced dilation homomorphism system  $(\tilde{\pi}, \tilde{S}, \tilde{T}, \tilde{W}) $ is irreducible.

\begin{definition} An homomorphism dilation system $(\pi, W, S, T)$ of a linear system $(\varphi, V)$ is called {\it linearly minimal} if $span\{\pi(\mathcal{A})TV\} = W$, and it is called  a {\it principle  dilation} if it  is both linearly minimal and irreducible.
\end{definition}

Let $(\pi,  S, T, W)$ be a homomorphism dilation system. Clearly, by replacing $W$ with $span\{\pi(\mathcal{A})TV\}$, we  get a linearly minimal dilation. Then, the reduced dilation system of the new linearly minimal dilation corresponding to the maximal invariant subspace  is irreducible. Therefore any homomorphism dilation system leads to a principle dilation system. In what follows we will focused only on linearly minimal dilations.

We construct two very special but important dilations for a given linear system that are essential for our structural theory of dilations. We first introduce  the canonical dilation. Let $(\varphi, \mathcal{A}, V)$ be a linear system.   For $a\in \mathcal{A}, x\in V$, define $\alpha_{a,x}\in L(\mathcal{A},V)$ by
$$\alpha_{a,x}(\cdot):=\varphi(\cdot a) x.$$ Let $W:=span\{\alpha_{a,x}:a\in  \mathcal{A},x\in V\}\subset L(\mathcal{A},V).$
Define $\pi_{c}: \mathcal{A}\rightarrow L(W)$ by
$\pi_{c}(a)(\alpha_{b,x}):=\alpha_{ab,x}.$ It is easy to see that
$\pi_{c}$ is a unital homomorphism. For $x\in V$ define $T:V
\rightarrow L(\mathcal{A}, V)$ by $T_x:=\alpha_{I,x}=\phi(\cdot
I)x=\phi(\cdot)x$. Define $S:W\rightarrow W$ by setting
$S(\alpha_{a,x}):=\phi(a)x$ and extending linearly to $W$. If
$a\in \mathcal{A}, x\in V$ are arbitrary, we have
$S\pi_{c}(a)Tx=S\pi_{c}(a)\alpha_{I,x}=S \alpha_{a,x}=\phi(a)x.$ Hence
$\varphi(a)=S\pi_{c}(a)T$  for all $a\in \mathcal{A}$. Thus $(\pi_{c}, S, T, W)$ is a dilation homomorphism of $(\varphi, V)$, and  we  will call it  the {\it canonical dilation} of $(\varphi, V)$.

By  the construction of $W$ and the definitions of $T$ and $\pi_{c}$ it is obvious that $(\pi_{c}, S, T, W)$ is a linearly minimal dilation. For the irreducibility, note that
$$
ker(S) = \{\sum_{i}c_{i}\alpha_{a_{i}, x_{i}}\in W:  \sum_{i}c_{i}\varphi(a_{i}) x_{i} =0 \}
$$

Let $w=\sum_{i}c_{i}\alpha_{a_{i}, x_{i}}\in ker S$. Then $\pi_{c}(a)w \in ker S$ for all $a\in \mathcal{A}$ if and only if  $\sum_{i}c_{i}\varphi(aa_{i})(x_{i}) = 0$ for all $a\in \mathcal{A}$, which in turn is equivalent to the condition $w = \sum_{i}c_{i}\alpha_{a_{i}, x_{i}} = 0$ as an element in $L(\mathcal{A}, V)$. Therefore $ker (S)$ does not contain any nontrivial $\pi_{c}$-invariant subspaces, and consequently we obtain:

\begin{proposition} The canonical dilation of a linear system $(\varphi, \mathcal{A}, V)$ is a principle dilation.
\end{proposition}

\begin{remark} In the case that $\varphi$ is already a unital homomorphism, the canonical dilation $\pi_{c}$ must be  $\varphi$. This can be easily seen by mapping $x\in V$ to $\alpha_{I, x}\in W$. Clearly this is well-defined and linear. The surjectivity follows from the fact that
$$
\alpha_{a, x}(b) = \varphi(ba)x = \varphi(b)\varphi(a)x = \alpha_{I, \varphi(a)x}(b)
$$
i.e., $\alpha_{a, x} = \alpha_{I, \varphi(a)x}$. With this identification it is easy to see that $S$ and $T$ constructed in the canonical dilation are inverse to each other.

\end{remark}

We will see in the next section that the canonical dilation is the one that has the ``smallest" dilation dimension and all the principle dilations are equivalent.
Note that for any linearly minimal dilation $(\pi, S, T, W)$ for a finite-dimensional system $(\varphi, \mathcal{A},  V)$, we always have $dim W \leq (dim \mathcal{A}) (dim V)$.    Now we construct a linearly minimal dilation which has the maximal dilation dimension $(dim \mathcal{A}) (dim V)$ , and we will show later that every linearly minimal dilation system is equivalent to a reduced dilation system of this  dilation. 
\vspace{3mm}

Let $W = \mathcal{A}\otimes V$.   Define $\pi_{u}: \mathcal{A} \rightarrow L(W)$, $S: W\rightarrow V$ and  $T: V \rightarrow W$ by the following:
$$
\pi_{u}(a) (\sum_{i}c_{i}a_{i}\otimes x_{i}) = \sum_{i}c_{i}(ab_{i})\otimes x_{i},
$$
$$
Tx = I\otimes x, \   \  \  \  \  \  S(\sum_{i}c_{i}a_{i}\otimes x_{i})  = \sum_{i}c_{i}\varphi(a_{i})x_{i}.
$$
Then $\pi_{u}$ is a homomorphism and  $$S\pi_{u}(a)Tx = S\pi_{u}(a)(I\otimes x) = S (a\otimes x) = \varphi(a)x$$ for all $x\in V$ and all $a\in \mathcal{A}$. Thus $(\pi_{u}, S, T, W)$ is a homomorphism dilation system of $(\varphi, V)$. Moreover, since $\pi_{u}(a)Tx = a\otimes x$, we have $span \{\pi_{u}(a)Tx:  a\in\mathcal{A}, x\in V\} = W$. Thus $(\pi_{u}, S, T, W)$ is a linearly minimal dilation system with the property $dim W = (dim \mathcal{A}) (dim V)$.

\begin{definition} The above constructed dilation $(\pi_{u}, S, T, W)$ is called the {\it universal dilation}  of $(\varphi, V)$.
\end{definition}

\section{The Structural Theorems}

In this section we present our main results about the classifications of all linearly minimal homomorphism dilations.

\begin{definition} Let $(\pi_{1}, S_{1}, T_{1}, W_{1})$ and $(\pi_{2}, S_{2}, T_{2}, W_{2})$ be two linearly minimal homomorphism dilation systems for a linearly system $(\varphi, V)$. We say that the two dilation homomorphism systems are  {\it equivalent} if there exists a bijective linear map $R: W_{1}\rightarrow W_{2}$ such that $RT_{1} = T_{2}$, $S_{2}R = S_{1}$ and $\pi_{1}(a) = R^{-1}\pi_{2}(a)R$ for all $a\in\mathcal{A}$,
\end{definition}

We first point out that $S_{2}R = S_{1}$ automatically follows from the other two conditions.

\begin{proposition} \label{lemma-1} Let $(\pi_{1}, S_{1}, T_{1}, W_{1})$ and $(\pi_{2}, S_{2}, T_{2}, W_{2})$ be two linearly minimal homomorphism dilation systems for a linearly system $(\varphi, V)$. If there exists a bijective linear map $R: W_{1}\rightarrow W_{2}$ such that $RT_{1} = T_{2}$ and $\pi_{1}(a) = R^{-1}\pi_{2}(a)R$ for all $a\in\mathcal{A}$, then $S_{2}R = S_{1}$ and hence the two systems are equivalent.
\end{proposition}

\begin{proof} Since the dilation systems are linearly minimal, we have that $W_{j} = span \  \pi(\mathcal{A})T_{j}(V)$ for $j=1, 2$.  So for any $w = \sum_{i}c_{i}\pi_{1}(a_{i})T_{1}x_{i}\in W_{1}$, we get

\begin{eqnarray*}
S_{2}R w &= &S_{2}(\sum_{i}c_{i}R\pi_{1}(a_{i})T_{1}x_{i}) = S_{2}(\sum_{i}c_{i}\pi_{2}(a_{i})RT_{1}x_{i}) \\
&= &\sum_{i}c_{i}S_{2}\pi_{2}(a_{i})T_{2}x_{i} = \sum_{i}c_{i} \varphi(a_{i}) x_{i} \\
&= &\sum_{i}c_{i}S_{1}\pi_{1}(a_{i})T_{1}x_{i} = S_{1}w.
 \end{eqnarray*}
 Thus $S_{2}R = S_{1}$.

\end{proof}

The following tells us all the  principle  homomorphism dilation systems are equivalent:

\begin{theorem} \label{equivalent} If $(\pi_{1}, S_{1}, T_{1}, W_{1})$ and $(\pi_{2}, S_{2}, T_{2}, W_{2})$ are two principle homomorphism dilation systems for $(\varphi, \mathcal{A}, V)$ , then $(\pi_{1}, S_{1}, T_{1}, W_{1})$ and $(\pi_{2}, S_{2}, T_{2}, W_{2})$  are equivalent.
\end{theorem}

\begin{proof} Since both dilations are linearly minimal, we have  $W_{i} = span \ \pi(\mathcal{A})T_{i}(V)$ for $i=1, 2$.  Define $R: W_{1} \rightarrow W_{2}$ by
$$
R(w) = \sum_{i}c_{i}\pi_{2}(a_{i})T_{2}(v_{i})
$$
if $w = \sum_{i}c_{i}\pi_{1}(a_{i})T_{1}(v_{i}$). In order for $T$ to be well-defined and induces the equivalence between $\pi_{1}$ and $\pi_{2}$, it suffices to show that $$w = \sum_{i}c_{i}\pi_{1}(a_{i})T_{1}(v_{i}) = 0$$ if and only if $$\sum_{i}c_{i}\pi_{2}(a_{i})T_{2}(v_{i}) = 0.$$ Assume to the contrary that $w \neq 0$.
Since
$$
S_{1}w = \sum_{i}c_{i}\varphi(a_{i})v_{i} = S_{2}\sum_{i}c_{i}\pi_{2}(a_{i})T_{2}(v_{i}) =S_{2}(0) = 0
$$
we get that $w\in ker(S_{1})$. Moreover,

\begin{eqnarray*}
S_{1}\pi_{1}(a)w &=& \sum_{i}c_{i}S_{1}\pi_{1}(aa_{i})T_{1}(v_{i})\\
&=& \sum_{i}c_{i}S_{1}\varphi(aa_{i})(v_{i})\\
&=& \sum_{i}c_{i}S_{2}\pi_{2}(aa_{i})T_{2}(v_{i})\\
&=& S_{2}\pi_{2}(a) \sum_{i}c_{i}\pi_{2}(a_{i})T_{2}(v_{i})\\
&=& S_{2}\pi_{2}(a)(0) = 0.
\end{eqnarray*}
Thus, $\pi_{1}(a)w\in ker(S_{1})$ for all $a\in\mathcal{A}$. So  $M = \{\pi_{1}(a)w: a\in\mathcal{A}\}$ is a nonzero $\pi_{1}$-invariant subspace insider $ker(S_{1})$, which leads to a contradiction since the  dilation $(\pi_{1}, S_{1}, T_{1}, W_{1})$ is irreducible. The argument for the other direction is the same. By the definition of $R$, we clearly have for  any $w = \sum_{i}c_{i}\pi_{1}(a_{i})T_{1}(v_{i}) \in W_{1}$ that
\begin{eqnarray*}
R\pi_{1}(a)w &=& R\sum_{i}c_{i}\pi_{1}(aa_{i})T_{1}(v_{i}) = \sum_{i}c_{i}\pi_{2}(aa_{i})T_{2}(v_{i}) \\
&=& \pi_{2}(a)\sum_{i}c_{i}\pi_{2}(a_{i})T_{2}(v_{i}) = \pi_{2}(a)Rw,
\end{eqnarray*}
and $RT_{1}(v) = T_{2}v$ for any $v\in V$. Thus we get $\pi_{1}(a) = R^{-1}\pi_{2}(a)R$ and $RT_{1} = T_{2}$, and therefore, by Proposition \ref{lemma-1}, we have that  $(\pi_{1}, S_{1}, T_{1}, W_{1})$ and $(\pi_{2}, S_{2}, T_{2}, W_{2})$  are equivalent.
\end{proof}

\begin{corollary}  Let $(\varphi, \mathcal{A}, V)$ be a linear system such that both $\mathcal{A}$ and $V$ are finite dimensional.

(i) Assume that $(\pi, S, T, W)$ is a principle dilation system of $(\varphi, V)$ such that $dim (W) = (dim \mathcal{A}) (dim (V)) $. Then any linearly minimal dilation system of $(\varphi, V)$ is irreducible, and hence a principle dilation system.

(ii) Assume that $(\pi, S, T, W)$ is a principle dilation system of $(\varphi, V)$. If $(\pi_{1}, S_{1}, T_{1}, W_{1})$ is a minimal dilation system of $(\varphi, V)$ such that $dim W_{1} \leq dim W$, then it is irreducible.
\end{corollary}

\begin{proof} (i) Let $(\pi_{1}, S_{1}, T_{1}, W_{1})$ be a linearly minimal dilation system of $(\varphi, V)$. Then $$dim W_{1} = dim span\{\pi_{1}(\mathcal{A})T_{1}V\} \leq  (dim \mathcal{A}) (dim (V))  = dim W.$$ Let  $(\tilde{\pi_{1}}, \tilde{S_{1}}, \tilde{T_{1}}, \tilde{W_{1}})$ be the reduced dilation system of $(\pi_{1}, S_{1}, T_{1}, W_{1})$ corresponding to the maximal $\pi_{1}$-invariant subspace of $ker(S_{1})$.  Then  $(\tilde{\pi_{1}}, \tilde{S_{1}}, \tilde{T_{1}}, \tilde{W_{1}})$  is both irreducible and linearly minimal. Thus it is a principle dilation system. By Theorem \ref{equivalent}, we get that $\pi$ and $\tilde{\pi_{1}}$ are equivalent, and hence
$dim \tilde{W_{1}} = dim W$. Since $dim \tilde{W}_{1} \leq dim W_{1} \leq dim W$, we obtain that $dim W_{1} = dim \tilde{W_{1}}$, which implies that $(\pi_{1}, S_{1}, T_{1}, W_{1})$ is irreducible.

(ii) Clearly the same argument above also works for part (ii).
\end{proof}

\begin{corollary} Let $(\pi_{1}, S_{1}, T_{1}, W_{1})$ be a linearly minimal dilation system of $(\varphi, V)$.

(i) If $ker(S_{1})$ does not contain any nonzero $\pi$-invariant subspaces, then $\pi_{1}$ is equivalent to the canonical homomorphism dilation $\pi_{c}$.

 (ii) Assume that  $dim W_{1} < \infty$. Then  if $\pi_{1}$ is equivalent to the canonical homomorphism $\pi_{c}$, then  $ker(S_{1})$ does not contain any nonzero $\pi$-invariant subspaces.
\end{corollary}

\begin{proof}  (i) If $ker (S_{1})$ does not contain any nonzero $\pi_{1}$-invariant subspaces, then by definition it is a principle dilations and hence is equivalent to $\pi_{c}$ by Theorem \ref{equivalent}.

(ii) Assume that $\pi_{1}$ is equivalent to the canonical homomorphism dilation $\pi_{c}$. Then $dim (W_{1}) = dim W$, where $W$ is the dilation space for the canonical dilation. Let $K$ be the largest $\pi$-invariant subspace contained in $ker (S_{1})$ and let $(\tilde{\pi_{1}}, \tilde{S_{1}}, \tilde{T_{1}}, W_{1}/K)$ be the reduced homomorphism dilation system. Then, by Theorem \ref{equivalent} again, $\tilde{\pi_{1}}$ and $\pi_{c}$ are equivalent homomorphisms, and so we get $dim (W) = dim (W_{1}/K)$. This implies that $dim (W_{1}) = dim (W_{1}/K)$. Thus $dim K = 0$ since $dim W_{1} < \infty$. Therefore $ker(S_{1})$ does not contain any nonzero $\pi$-invariant subspaces.
\end{proof}

\begin{remark} We don't know if (ii) is still true when $dim W_{1}$ is not finite dimensional.
\end{remark}

The term of ``universal dilation" is justified by the following:

\begin{theorem} \label{structure-1}  Any linearly minimal homomorphism dilation of a linear system $(\varphi, V)$ is equivalent to a reduced  homomorphism dilation system of its universal dilation.
\end{theorem}
\begin{proof} Let $(\pi_{1}, S_{1}, T_{1}, W_{1})$ be a linearly minimal dilation system.  Define $K$ by
$$
K = \{w = \sum_{i}c_{i}a_{i}\otimes x_{i}: \sum_{i}c_{i}\pi_{1}(a_{i})T_{1}x_{i}= 0\}.
$$
Claim: $K$ is a $\pi_{u}$-invariant subspace contained in $ker (S)$. In fact, if  $w = \sum_{i}c_{i}a_{i}\otimes x_{i}\in K$, then
\begin{eqnarray*}
Sw &=& \sum_{i}c_{i}S(a_{i}\otimes x_{i}) = \sum_{i}c_{i}S\pi_{u}(a_{i})Tx\\
& =& \sum_{i}c_{i}\varphi(a_{i}) x_{i} =  \sum_{i}c_{i}S_{1}\pi_{1}(a_{i})T_{1}x_{i}  \\
&=& S_{1}\sum_{i}c_{i}\pi_{1}(a_{i})T_{1}x_{i} = S(0) = 0 .
\end{eqnarray*}
Thus $K\subseteq ker (S)$. Moreover, for any $a\in \mathcal{A}$ and $w = \sum_{i}c_{i}a_{i}\otimes x_{i}\in K$, we have
$$
 \sum_{i}c_{i}\pi_{1}(aa_{i})T_{1}x_{i} = \pi_{1}(a) \sum_{i}c_{i}\pi_{1}(a_{i})T_{1}x_{i} = 0.
 $$
 Thus $\pi_{u}(a)w =  \sum_{i}c_{i}(aa_{i})\otimes x_{i} \in K$. Therefore $K$ is a $\pi_{u}$-invariant subspace contained in $ker (S)$.

 Let $(\tilde{\pi_{u}}, \tilde{S}, \tilde{T}, W/K)$ be the reduced dilation homomorphism.   Define $R: W/K \rightarrow W_{1}$ by
 $$
 R[w] = \sum_{i}c_{i}\pi_{1}(a_{i})T_{1}x_{i}
 $$
 for any $[w]\in W/K$ represented by $w = \sum_{i}c_{i}a_{i}\otimes x_{i}$ . Then, by the definition of $K$, we have $R[w] =  \sum_{i}c_{i}a_{i}\otimes x_{i} = 0$ if and only if $w\in K$. Hence, $R$ is a well defined injective linear map. Clearly it is also surjective since $span\{\pi_{1}(\mathcal{A})T_{1}V\} = W_{1}$. Moreover, for any $w = \sum_{i}c_{i}a_{i}\otimes x_{i}\in W$, we have
 \begin{eqnarray*}
\pi_{1}(a)R([w]) &=& \pi_{1}(a)\sum_{i}c_{i}\pi_{1}(a_{i})T_{1}x_{i} = \sum_{i}c_{i}\pi_{1}(aa_{i})T_{1}x_{i}\\
& = & R([\sum_{i}c_{i}(aa_{i}) \otimes x_{i}] = R\tilde{\pi_{u}}(a)([w])
\end{eqnarray*}
Thus $\pi_{1}(a) = R\tilde{\pi_{u}}(a)R^{-1}$ for any $a\in\mathcal{A}$. Moreover, for any $w = \sum_{i}c_{i}a_{i}\otimes x_{i}\in W$  we have
$$
R\tilde{T}x = R[Tx] = R[I\otimes x] = \pi_{1}(I)T_{1}x.
$$
Hence $R\tilde{T} = T_{1}$. Therefore $(\pi_{1}, S_{1}, T_{1}, W_{1})$ and $(\tilde{\pi_{u}}, \tilde{S}, \tilde{T}, W/K)$ are equivalent.
\end{proof}

In order to classify the linearly minimal homomorphism dilation systems we introduce the following:

\begin{definition} Let  $(\pi_{u}, S, T, W)$ be the universal dilation system  and $(\pi_{1}, S_{1}, T_{1}, W_{1})$ be a linearly minimal homomorphism dilation system for a linear system $(\varphi, V)$. Then the $\pi_{u}$-invariant subspace $K_{1}$ introduced in the above proof will be called the {\it reduced  subspace} associated with  $(\pi_{1}, S_{1}, T_{1}, W_{1})$.
\end{definition}

\begin{remark} We point out that the reduced subspace $K$ of a  linearly minimal  homomorphism dilation system $(\pi_{1}, S_{1}, T_{1}, W_{1})$ is different from the maximal $\pi_{1}$-invariant subspace  $M$ contained in $ker(S_{1})$ which is used to reduce  $(\pi_{1}, S_{1}, T_{1}, W_{1})$ to the ``smallest" dilation --- the principle dilation, while $K$ is a $\pi_{u}$-invariant subspace contained in the universal dilation space $W$ (i.e.,  $\mathcal{A}\otimes V$) that is used to reduce the universal dilation system to $(\pi_{1}, S_{1}, T_{1}, W_{1})$.

\end{remark}

The following gives us a classification of all linearly minimal homomorphism dilations systems for a given linear system. 

\begin{theorem} \label{class-1}  Let $K_{1}$ and $K_{2}$ be the reduced subspaces associated with the linearly minimal homomorphism dilation systems $(\pi_{1}, S_{1}, T_{1}, W_{1})$ and $(\pi_{2}, S_{2}, T_{2}, W_{2})$, respectively. Then the two homomorphism dilation systems $(\pi_{1}, S_{1}, T_{1}, W_{1})$ and $(\pi_{2}, S_{2}, T_{2}, W_{2})$ are equivalent if and only if $K_{1} =K_{2}$.
\end{theorem}
\begin{proof} By Theorem \ref{structure-1} we only need to prove that if $(\pi_{1}, S_{1}, T_{1}, W_{1})$ and $(\pi_{2}, S_{2}, T_{2}, W_{2})$ are equivalent, then $K_{1} =K_{2}$.  Let $R: W_{1} \rightarrow W_{2}$  be a bijective linear map such that $\pi_{2}(a)R = R\pi_{1}(a)$ for all $a\in\mathcal{A}$, $S_{2}R = S_{1}$ and $RT_{1} = T_{2}$.

Let $w = \sum_{i}c_{i}a_{i}\otimes x_{i}\in W$.  Since $$\sum_{i}c_{i}\pi_{1}(a_{i})T_{1}x_{i} = R^{-1}\sum_{i}c_{i}\pi_{2}(a_{i})RT_{1}x_{i} = R^{-1}\sum_{i}c_{i}\pi_{2}(a_{i})T_{2}x_{i},$$ we get that $\sum_{i}c_{i}\pi_{i}(a_{i})x_{i} = 0$ if and only if $\sum_{i}c_{i}\pi_{2}(a_{i})T_{2}x_{i} = 0$, i.e., $w\in K_{1}$ if and only if $w\in K_{2}$. Hence $K_{1} = K_{2}$.
\end{proof}

The above theorem shows that the equivalent class of linearly minimal homomorphism dilation systems is uniquely determined by the reduced subspace.We will show by example in section 4 that there could be infinitely many inequivalent linearly minimal homomorphism dilation systems even in the finite-dimensional case (i.e., $dimV < \infty$ and $dim(\mathcal{A}) < \infty$). Additionally,  there is a weaker version of equivalence which seems also relevant to the dilation theory: If $(\pi_{1}, S_{1}, T_{1}, W_{1})$ be a linearly minimal dilation system for a linear system $(\varphi, V)$, and $\pi_{2}$ is a homomorphsim from $\mathcal{A}$ to $L(W_{2})$ such that $\pi_{1}$ and $\pi_{2}$ are equivalent in the usual sense, i.e. $\pi_{1}(a) = R^{-1}\pi_{2}(a)R$ $(\forall a\in\mathcal{A})$ for some isomorphism $R: W_{1}\rightarrow W_{2}$, then $(\pi_{2}, S_{2}, T_{2}, W_{2})$ is an equivalent dilation system with $S_{2} = S_{1}R^{-1}$ and $T_{2}= RT_{1}$. Thus it is interesting to known that under what condition do we have two equivalent homomorphisms $\pi_{1}$ and $\pi_{2}$ for linearly minimal homomorphism dilation systems  $(\pi_{1}, S_{1}, T_{1}, W_{1})$ and $(\pi_{2}, S_{2}, T_{2}, W_{2})$. For this purpose we introduce the following concept of equivalence  for the  reducing invariant subspaces.

\begin{definition} Let $\pi_{u}, S, T, W)$ be  the universal dilation system of a linearly system $(\varphi, V)$. 
Two $\pi_{u}$-invariant subspaces $K_{1}$ and $K_{2}$ of $ker (S)$ are called {\it strongly isomorphic } if there is an isomorphism $R: W\rightarrow W$ such that $R(K_{1}) = K_{2}$ and $\pi_{u}(a)Rw- R\pi_{u}(a)w \in K_{2}$ for all $a\in\mathcal{A}$ and all $w\in W$, i.e., the quotient maps of $\pi_{u}(a)$ and $R$ on $W/K_{2}$ commute for all $a\in\mathcal{A}$.

\end{definition}

\begin{theorem} \label{class-2} Let $K_{1}$ and $K_{2}$ be the reduced subspaces for the linearly minimal homomorphism dilation systems $(\pi_{1}, S_{1}, T_{1}, W_{1})$ and $(\pi_{2}, S_{2}, T_{2}, W_{2})$, respectively. Then $\pi_{1}$ and $\pi_{2}$ are equivalent if and only if $K_{1}$ and $K_{2}$ are strongly isomorphic.
\end{theorem}

\begin{proof} By Theorem \ref{structure-1} we can assume that $(\pi_{i}, S_{i}, T_{i}, W_{i})$ is the reduced homomorphism dilation of the universal dilation associated with $K_{i} (i =1, 2)$.

$(\Leftarrow)$:  Assume that $K_{1}$ and $K_{2}$ are strongly isomorphic. Then there is an isomorphism $R: W\rightarrow W$ such that $R(K_{1}) = K_{2}$ and $\pi_{u}(a)Rw- R\pi_{u}(a)w \in K_{2}$ for all $a\in\mathcal{A}$ and all $w\in W$. Let $\tilde{R}: W_{1} = W/K_{1} \rightarrow W/K_{2} = W_{2}$ be defined by
$$
\tilde{R}[w] = [Rw], \  \  \ w\in W,
$$
where we use $[\cdot ]$ to denote the element in the corresponding quotient space. Then $\tilde{R}$ is a bijective linear transformation. Note that since $\pi_{2}$ is the reduced homomorphism of $\pi_{u}$ on $W/K_{2}$, we have that $\pi_{2}(a)\tilde{R}([w]) = \pi_{2}(a)[Rw] = [\pi_{u}(a)Rw].$ Similarly, $\tilde{R}\pi_{1}(a)[w] = \tilde{R}[\pi_{u}(a)w] = [R\pi_{u}(a)w]$. Thus, from  $\pi_{u}(a)Rw- R\pi_{u}(a)w \in K_{2}$, we obtain that $\pi_{2}(a)\tilde{R}[w] = \tilde{R}\pi_{1}(a)[w]$, which implies that $\pi_{1}$ and $\pi_{2}$ are equivalent.

$(\Rightarrow)$:  Assume that $\pi_{1}$ and $\pi_{2}$ are equivalent.Then there is bijective linear map $L: W/K_{1} \rightarrow W/K_{2}$ such that
$\pi_{2}(a)L = L\pi_{1}(a)$ for all $a\in\mathcal{A}$. Since  $dim (K_{1}) = dim (K_{2})$, we obtain that there exists a bijective linear map $R: W \rightarrow W$ such that
the $R(K_{1}) = K_{2}$ and the induced quotient map $\tilde{R}$ is $L$. Moreover, from $\pi_{2}(a)L = L\pi_{1}(a)$ we have that $\pi_{2}(a)\tilde{R}= \tilde{R}\pi_{1}(a)$, which is equivalent to the condition that $\pi_{u}(a)Rw- R\pi_{u}(a)w \in K_{2}$ for all $a\in\mathcal{A}$ and all $w\in W$. Thus $K_{1}$ and $K_{2}$ are strongly isomorphic.
\end{proof}

\section{Remarks and Examples}

Theorem \ref{class-1} and Theorem \ref{class-2} provide us with two classifications for linearly minimal  homomorphism dilations based on the universal dilation invariant subspaces in the kernel of the map $S: \mathcal{A}\otimes V \rightarrow V$ defined by $S(a\otimes x) = \varphi(a)x$.
 These lead to many interesting questions, especially in the finite dimensional case.  For example, (1) under what condition on $(\varphi, \mathcal{A}, V)$ do we have the property that for every  $k$ between the dimensions of $V$ and $\mathcal{A}\otimes V$ there exists a linearly minimal dilation with dilation dimension $k$. (2) When do we have only finite many inequivalent linearly minimal homomorphism dilations? (Examples 4.5 and 4.7 show that we could have infinitely many inequivalent classes even both $\mathcal{A}$ and $V$ are finite dimensional.)  (3) Under what condition do we have that the principle and universal dilations are the only two classes of linearly minimal dilations? (4) We will construct an example showing that there exist reduced subspaces  $K_{1}$ and $K_{2}$  that are strongly isomorphic by $K_{1} \neq K_{2}$. However, it would be interesting to know that if the condition $dim K_{1} = dim K_{2}$ automatically implies that they are strongly isomorphic.

 In what follows we will answer some of these questions and at the same time constructing some examples showing the complexity of other questions.

 Let $M = \{\sum_{i}c_{i}a_{i}x_{i}: \sum_{i}c_{i}\varphi(aa_{i})x_{i}=0, \forall a\in \mathcal{A}\}$. Then $M$ is the largest $\pi_{u}$-invariant subspace contained in $ker(S)$. Hence, by Theorem \ref{equivalent} we have that the universal homomorphism dilation equivalent to the principle dilation if and only if $M =  \{0\}$. Moreover we  have

 \begin{proposition} A linear system  $(\varphi, V)$  has only one equivalent class of linearly minimal homomorphism dilations if and only if $M = \{0\}$.
  \end{proposition}

  \begin{proof} Let $(\pi_{1}, S_{1}, T_{1}, W_{1})$ be a linearly minimal dilation homomorphism system for $(\varphi, V)$. Let $K_{1}$ be its reduced subspace. If $w = \sum_{i}c_{i}a_{i}x_{i}\in K_{1}$, then $\sum_{i}c_{i}\phi_{1}(aa_{i})T_{1}x_{i} = 0$ for every $a\in\mathcal{A}$. Since $\varphi(\cdot) = S_{1}\pi_{1}(\cdot)T_{1}$, we get that $\sum_{i}c_{i}\varphi(aa_{i})x_{i} = 0$ for all $a\in \mathcal{A}\}$, i.e., $w\in M$. Thus $K_{1} = \{0\}$, and so  $(\pi_{1}, S_{1}, T_{1}, W_{1})$ is equivalent to the universal dilation.
    \end{proof}

\begin{corollary} Let $(\varphi, \mathcal{A}, V)$ be a linear system. If $ker(\varphi)$ contains a proper left ideal, then the universal dilation is not equivalent to its principle dilation.
\end{corollary}
\begin{proof} Let $a$ be a nonzero element in the left ideal. Then for any $x\in V$ and any $b\in \mathcal{A}$ we have $\varphi(ba)x = 0$, which implies that $a\otimes x\in M$. Hence $M \neq \{0\}$ and consequently the universal dilation is not equivalent to the principle dilation.
\end{proof}

Note that  if $dim(V) = 1$, then $\mathcal{A}\otimes V = \{a\otimes x: a\in \mathcal{A}\}$, where $x$ is a fixed nonzero vector in $V$. So $M = \{a\otimes x: \varphi(ba)x = 0, \forall b\in\mathcal{A}\} =  \{a\otimes x: \varphi(ba)= 0, \forall b\in\mathcal{A}\}$, where we used the factor that $\varphi(ba)$ is a scalar. Thus we get

\begin{corollary} Let $(\varphi, \mathcal{A}, V)$ be a linear system such that $dim(V) = 1$. Then its universal dilation and  principle dilation are equivalent if and only if $ker(\varphi)$ does not contain any proper left ideals.
\end{corollary}


 \begin{example} Let $\mathcal{A} = \M_{n}$ be the $n\times n$ matrix algebra, and $\varphi(A) = {1\over n}tr(A)$. Then it is easy to show that $ker(\varphi)$ does not contain any proper left ideals, and hence the universal dilation is the same as its canonical dilation.  For example if $n=2$, then $\varphi(A) = \frac{a+d}{2}$, where
 $$
 A =  \left(\begin{array}{cc}
                     a & b \\
                     c & d \\
                   \end{array}
                 \right)
                 $$
                 Then the canonical (as well as the universal)  homomorphism dilation system $(\pi, S, T, \Bbb{C}^{4})$ is given by
 \[ \pi(A) =\left(
                                    \begin{array}{cccc}
                                      a & b & 0 & 0 \\
                                      c & d & 0 & 0 \\
                                      0 & 0 & a & b \\
                                      0 & 0 & c & d \\
                                    \end{array}
                                  \right)
                  \]
 with
 \[S=\left(
       \begin{array}{cccc}
         \frac{1}{2} & -\frac{1}{2} & \frac{1}{2} & \frac{1}{2} \\
       \end{array}
     \right) \ \ \mbox{ and } \ \ T=\left(
                                      \begin{array}{cccc}
                                        \frac{1}{2} &
                                        -\frac{1}{2} &
                                        \frac{1}{2} &
                                        \frac{1}{2}
                                      \end{array}
                                    \right)^{t}.
 \]
 \end{example}

  \begin{example} Let $\mathcal{A} = \mathcal{T}_{n}$ be the algebra of all the $n\times n$ upper triangular  matrices, $\mathcal{T}_{n, 0}$ be the algebra of all the $n\times n$ strictly  upper triangular matrices, and $\varphi(a) = {1\over n}tr(a)$. Then $\mathcal{T}_{n, 0}$ is a proper ideal contained in $ker(\varphi)$. Thus the universal dilation system is not equivalent to  its canonical dilation system.

  (i) For $n=2$ we have $\varphi(A) = {1\over2}(a+c)$ where
$$
A = \left(\begin{array}{cc}
                      a & b \\
                      0 & c \\
                    \end{array}
                  \right).
                  $$

 Then the universal dilation system $(\pi_{u}, S_{u}, T_{u}, \Bbb{C}^{3})$ and the canonical dilation system $(\pi_{c}, S_{c}, T_{c}, \Bbb{C}^{2})$ are given by:
$$ \pi_u(A)=\left(
             \begin{array}{ccc}
               a & 0 & 0 \\
               0 & a & b \\
               0 & 0 & c \\
             \end{array}
           \right), \ \ \ \ \ \ \mbox{and} \ \ \   \pi_c(A)=\left(
             \begin{array}{cc}
               a & 0 \\
               0 & c \\
             \end{array}
           \right),
$$
where
    \[S_u=\left(
            \begin{array}{ccc}
              \frac{1}{2} & 0 & \frac{1}{2} \\
            \end{array}
          \right),  \ \
          T_u=\left(
                \begin{array}{ccc}
                  1 &
                  0 &
                  1
                \end{array}
              \right)^{t},
    \]
     and
    \[S_c=\left(
            \begin{array}{cc}
              \frac{1}{2} & \frac{1}{2} \\
            \end{array}
          \right),  \ \
          T_c=\left(
                \begin{array}{cc}
                  1 &
                  1
                \end{array}
              \right)^{t}.
    \]
These are the only two linearly minimal homomorphism dilations.

    (ii) For the case $n=3$, we have $\varphi(A) = {1\over 3}(a+d+f)$ where
 $$
 A=\left(
                       \begin{array}{ccc}
                         a & b & c \\
                         0 & d & e \\
                         0 & 0 & f \\
                       \end{array}
                     \right).
                     $$

In this case the universal dilation system $(\pi_{u}, S_{u}, T_{u}, \Bbb{C}^{6})$ and the canonical dilation system $(\pi_{c}, S_{c}, T_{c}, \Bbb{C}^{3})$ are given by:
 $$
  \pi_{u}(A)=\left(
               \begin{array}{cccccc}
                 a & 0 & 0 & 0 & 0 & 0 \\
                 0 & a & b & 0 & 0 & 0 \\
                 0 & 0 & d & 0 & 0 & 0 \\
                 0 & 0 & 0 & a & b & c \\
                 0 & 0 & 0 & 0 & d & e \\
                 0 & 0 & 0 & 0 & 0 & f \\
               \end{array}
             \right),  \ \ \ \ \mbox{and} \  \ \pi_{c}(A) = \left(
                       \begin{array}{ccc}
                         a & 0 & 0 \\
                         0 & d & 0 \\
                         0 & 0 & f \\
                       \end{array}
                     \right),
                     $$
                     where
      \[ S_u=\left(
               \begin{array}{cccccc}
                 \frac{1}{3} & 0 & \frac{1}{3} & 0 & 0 & \frac{1}{3} \\
               \end{array}
             \right),  \ \ \ \ \
             T_u=\left(
                   \begin{array}{cccccc}
                     1 &
                     0 &
                     1 &
                     0 &
                     0 &
                     1
                   \end{array}
                 \right)^{t},
       \]
   and
     \[S_c=\left(
             \begin{array}{ccc}
               \frac{1}{3} & \frac{1}{3} & \frac{1}{3} \\
             \end{array}
           \right),  \ \ \ \
     T_c=\left(
           \begin{array}{ccc}
             1 &
             1 &
             1
           \end{array}
         \right)^{t}.
     \]

     In order to identify the rest of the equivalent classes of homomorphism dilations, we need to identify all the $\pi_{u}$-invariant subspaces in $ker(S_{u})$. Note that  $\ker (S_u)=span\{e_2, e_4, e_5\}$, and it is easy to verify that the
    maximal $\pi_u$-invariant subspace is $span\{e_2, e_4\}$, and any one-dimensional subspace of $span\{e_2, e_4\}$ is also $\pi_{u}$-invariant. Hence, by Theorem \ref{class-1},  we only have one equivalent class of $4$-dimensional homomorphism dilation systems, and infinitely many oinequivalent class of $5$-dimensional homomorphism dilation systems.

    The $4$-dimensional  equivalent class of homomorphism dilation systems is represented by $(\pi_{4}, S_{4}, T_{4}, \Bbb{C}^{4})$:
               $$
               S_{4}= \left(
                         \begin{array}{cccc}
                           \frac{1}{3} & \frac{1}{3} & 0 & \frac{1}{3} \\
                         \end{array}
                       \right), \ \
                      \pi_{4}(A) =  \left(
                         \begin{array}{cccc}
                           a & 0 & 0 & 0 \\
                           0 & d & 0 & 0 \\
                           0 & 0 & d & e \\
                           0 & 0 & 0 & f \\
                         \end{array}
                       \right), \  \
                      T_{4} =  \left(
                         \begin{array}{cccc}
                           1 &
                           1 &
                           0 &
                           1
                         \end{array}
                       \right)^{t}.
                     $$

 Two classes of homomorphism dilations systems represented by $(\pi_{5,1}, S_{5,1}, T_{5,1}, \Bbb{C}^{5})$ and $(\pi_{5,2}, S_{5,2}, T_{5,2}, \Bbb{C}^{5})$  that associated with $\pi_u$-invariant subspaces $K_{1} = span\{e_{2}\}$ and $K_{2} = span\{e_{4}\}$, respectively, are given by:
   $$
   S_{5, 1} =\left(
                         \begin{array}{ccccc}
                           \frac{1}{3} & \frac{1}{3} & 0 & 0 & \frac{1}{3} \\
                         \end{array}
                       \right), \ \
                      \pi_{5, 1}(A) = \left(
                         \begin{array}{ccccc}
                           a & 0 & 0 & 0 & 0\\
                           0 & d & 0 & 0 & 0 \\
                           0 & 0 & a & b & c \\
                           0 & 0 & 0 & d & e \\
                           0 & 0 & 0 & 0 & f \\
                         \end{array}
                       \right), \ \
                      T_{5} =  \left(
                         \begin{array}{ccccc}
                           1 &
                           1 &
                           0 &
                           0 &
                           1
                         \end{array}
                       \right)^{t}
        $$
       and
       $$
               S_{5,2}=\left(
                         \begin{array}{ccccc}
                           \frac{1}{3} & 0 & \frac{1}{3} & 0 & \frac{1}{3} \\
                         \end{array}
                       \right), \ \
                      \pi_{5,2}(A) = \left(
                         \begin{array}{ccccc}
                           a & 0 & 0 & 0 & 0\\
                           0 & a & b & 0 & 0 \\
                           0 & 0 & d & 0 & 0 \\
                           0 & 0 & 0 & d & e \\
                           0 & 0 & 0 & 0 & f \\
                         \end{array}
                       \right), \ \
                       T_{5,2}=\left(
                         \begin{array}{ccccc}
                           1 &
                           0 &
                           1 &
                           0 &
                           1
                         \end{array}
                         \right)^{t}
             $$
We leave the construction of the homomorphism dilation associated with the $\pi_{u}$-invariant subspace $K_{\alpha, \beta} = span\{\alpha e_{2} + \beta e_{4}\}$ for the interested readers.
 \end{example}

  \begin{example}  Let $\varphi:\M_2\to\M_2$ be defined by
\[
\varphi\left(\left(
            \begin{array}{cc}
              \alpha_1 & \alpha_2 \\
              \alpha_3 & \alpha_4 \\
            \end{array}
          \right)
\right)=\left(
          \begin{array}{cc}
            \displaystyle\sum_{i=1}^4 a_i\alpha_i & \displaystyle\sum_{i=1}^4 b_i\alpha_i \\
            \displaystyle\sum_{i=1}^4 c_i\alpha_i & \displaystyle\sum_{i=1}^4 d_i\alpha_i \\
          \end{array}
        \right)
.
\] Then we have the universal dilation $\pi_{u}:\M_2\to\M_8$ given by
\[
\pi_{u}\left(\left(
            \begin{array}{cc}
              \alpha_1 & \alpha_2 \\
              \alpha_3 & \alpha_4 \\
            \end{array}
          \right)\right)=\left(
                           \begin{array}{cccccccc}
                             \alpha_1 & \alpha_2 & 0 & 0 & 0 & 0 & 0 & 0 \\
                             \alpha_3 & \alpha_4 & 0 & 0 & 0 & 0 & 0 & 0 \\
                             0 & 0 & \alpha_1 & \alpha_2 & 0 & 0 & 0 & 0 \\
                             0 & 0 & \alpha_3 & \alpha_4 & b & 0 & 0 & 0 \\
                             0 & 0 & 0 & 0 & \alpha_1 & \alpha_2 & 0 & 0 \\
                             0 & 0 & 0 & 0 & \alpha_3 & \alpha_4 & e & 0 \\
                             0 & 0 & 0 & 0 & 0 & 0 & \alpha_1 & \alpha_2 \\
                             0 & 0 & 0 & 0 & 0 & 0 & \alpha_3 & \alpha_4 \\

                           \end{array}
                         \right)
\] with
$$
S_{u}=
       \left(
                    \begin{array}{cccccccc}
                      a_1 & a_3 & a_2 & a_4 & b_1 & b_3 & b_2 & b_4 \\
                      c_1 & c_3 & c_2 & c_4 & d_1 & d_3 & d_2 & d_4 \\
                    \end{array}
                  \right)
$$
and
$$
     T_{u}=\left(
         \begin{array}{cccccccc}
           1 & 0 &0&1&0&0&0&0\\
           0 & 0 &0&0&1&0&0&1
         \end{array}
       \right).
$$
\end{example}

Let $\tau$ and $\sigma$ the linear maps on $\M_2$ defined by $\tau(A) = A^{t}$ and $\sigma(A) = {1\over 2}Tr(A)I$, where $I\in\M_{2}$ is the identity matrix. Then it can be shown that there is no nontrivial $\pi_{u}$-invariant subspaces in $ker(S_{u})$, and thus the above formula also gives us the canonical dilation. The situation becomes quite different for transpose map of triangular matrices. For simplicity let us examine the transpose map on $\mathcal{T}_{2}$ and $\mathcal{T}_{3}$.

  \begin{example} Let $\tau:\mathcal{T}_2\to\M_2$ be the transpose map. Then  the universal dilation  system is given by:
\[ 
\pi_u\left(\left(
\begin{array}{cc}
 a & b  \\
 0 & c
\end{array}
\right)\right)=\left(
\begin{array}{cccccc}
 a & 0 & 0 & 0 & 0 & 0 \\
 0 & a & b & 0 & 0 & 0 \\
 0 & 0 & c & 0 & 0 & 0 \\
 0 & 0 & 0 & a & 0 & 0 \\
 0 & 0 & 0 & 0 & a & b \\
 0 & 0 & 0 & 0 & 0 & c \\
\end{array}
\right)
\] with
\[ S_u=\left(
       \begin{array}{cccccc}
         1 & 0 & 0 & 0 & 0 & 0 \\
         0 & 1 & 0 & 0 & 0 & 1 \\
       \end{array}
     \right) \ \
     \mbox{ and } \ \
     T_u=\left(
         \begin{array}{cccccc}
           1 & 0 &1&0&0&0\\
           0 & 0 &0&1&0&1
         \end{array}
       \right)^{t}.
 \]
 The canonical dilation system is given by:
\[
\pi_c\left(\left(
\begin{array}{cc}
 a & b  \\
 0 & c
\end{array}
\right)\right)=\left(
\begin{array}{cccc}
 a & 0 & 0 & 0 \\
 0 & c & 0 & 0 \\
 0 & 0 & a & b \\
 0 & 0 & 0 & c \\
\end{array}
\right)
\] with
\[ S_c=\left(
       \begin{array}{cccc}
         1 & 0 & 0 & 0 \\
         0 & 1 & 1 & 0 \\
       \end{array}
     \right) \ \
     \mbox{ and } \ \
     T_c=\left(
         \begin{array}{cccc}
           1 & 0 &0&1\\
           0 & 1 &0&0
         \end{array}
       \right)^{t}.
 \]

 Furthermore, we have $$\ker S_u=span \{e_2-e_6, e_3, e_4, e_5\}.$$ In $\ker S_u$, the maximal $\pi_u$-invariant subspace
 is $M = span\{e_4, e_5\}$,  and for any given $\alpha, \beta$, the one-dimensional subspace $K_{\alpha, \beta} = span\{\alpha e_{4}+\beta e_{5}\}$ is $\pi_{u}$-invariant. So, again, there are infinitely many inequivalent classes of $5$-dimensional dilations. The two special ones corresponding to $K_{1, 0}$ and $K_{0, 1}$ are represented by:

The
 \[ \left(
      \begin{array}{ccccc}
        1 & 0 & 0 & 0 & 0 \\
        0 & 1 & 0 & 0 & 1 \\
      \end{array}
    \right)\left(
             \begin{array}{ccccc}
               a & 0 & 0 & 0 & 0 \\
               0 & a & b & 0 & 0 \\
               0 & 0 & c & 0 & 0 \\
               0 & 0 & 0 & a & b \\
               0 & 0 & 0 & 0 & c \\
             \end{array}
           \right)\left(
                    \begin{array}{cc}
                      1 & 0 \\
                      0 & 0 \\
                      1 & 0 \\
                      0 & 0 \\
                      0 & 1 \\
                    \end{array}
                  \right) \ \ \mbox{ and }
  \]
  \[ \left(
      \begin{array}{ccccc}
        1 & 0 & 0 & 0 & 0 \\
        0 & 1 & 0 & 0 & 1 \\
      \end{array}
    \right)\left(
             \begin{array}{ccccc}
               a & 0 & 0 & 0 & 0 \\
               0 & a & b & 0 & 0 \\
               0 & 0 & c & 0 & 0 \\
               0 & 0 & 0 & a & 0 \\
               0 & 0 & 0 & 0 & c \\
             \end{array}
           \right)\left(
                    \begin{array}{cc}
                      1 & 0 \\
                      0 & 0 \\
                      1 & 0 \\
                      0 & 1 \\
                      0 & 1 \\
                    \end{array}
                  \right). \]

 (ii) Let $\tau:\mathcal{T}_3\to\M_3$ be the transpose map
\[
\tau\left(\left(
            \begin{array}{ccc}
              a & b & c \\
              0 & d & e \\
              0 & 0 & f \\
            \end{array}
          \right)
\right)=\left(
          \begin{array}{ccc}
            a & 0 & 0 \\
            b & d & 0 \\
            c & e & f \\
          \end{array}
        \right)
.
\] Then we have the canonical dilation $\pi_c:\cT_3\to\M_{10}$ by
\[
\pi_c\left(\left(
            \begin{array}{ccc}
              a & b & c \\
              0 & d & e \\
              0 & 0 & f \\
            \end{array}
          \right)\right)=\left(
                           \begin{array}{cccccccccc}
                             a & 0 & 0 & 0 & 0 & 0 & 0 & 0 & 0 & 0 \\
                             0 & d & 0 & 0 & 0 & 0 & 0 & 0 & 0 & 0 \\
                             0 & 0 & f & 0 & 0 & 0 & 0 & 0 & 0 & 0 \\
                             0 & 0 & 0 & a & b & 0 & 0 & 0 & 0 & c \\
                             0 & 0 & 0 & 0 & d & 0 & 0 & 0 & 0 & e \\
                             0 & 0 & 0 & 0 & 0 & d & e & 0 & 0 & 0 \\
                             0 & 0 & 0 & 0 & 0 & 0 & f & 0 & 0 & 0 \\
                             0 & 0 & 0 & 0 & 0 & 0 & 0 & a & b & 0 \\
                             0 & 0 & 0 & 0 & 0 & 0 & 0 & 0 & d & 0 \\
                             0 & 0 & 0 & 0 & 0 & 0 & 0 & 0 & 0 & f \\
                           \end{array}
                         \right)
\] with
\[ S_c=
       \left(
                    \begin{array}{cccccccccc}
                      1 & 0 & 0 & 0 & 0 & 0 & 0 & 0 & 0 & 0 \\
                      0 & 1 & 0 & 0 & 0 & 0 & 0 & 1 & 0 & 0 \\
                      0 & 0 & 1 & 1 & 0 & 1 & 0 & 0 & 0 & 0 \\
                    \end{array}
                  \right)
      \ \
     \mbox{ and } \ \
     T_c=\left(
         \begin{array}{ccc}
           1 & 0 & 0 \\
           0 & 1 & 0 \\
           0 & 0 & 1 \\
           0 & 0 & 0 \\
           0 & 0 & 0 \\
           0 & 0 & 0 \\
           0 & 1 & 0 \\
           0 & 0 & 0 \\
           1 & 0 & 0 \\
           1 & 0 & 0 \\
         \end{array}
       \right) .
 \]

  \end{example}

\begin{example} Let $v:\M_2\to\M_2$ be the linear map
\[
v\left(\left(
\begin{array}{cc}
a & b  \\
c & d
\end{array}
\right)\right)=\left(
\begin{array}{cc}
 \alpha_1(\xi_1 a+\xi_2 b)+\alpha_2(\xi_1 c+\xi_2 d) & 0  \\
 0 & \beta_1(\gamma_1 a +\gamma_2 b)+\beta_2(\gamma_1 c+\gamma_2 d)
\end{array}
\right).
\] Then we have the (linearly minimal) dilation $\pi:\cT_2\to\M_4$ by
\[
\pi\left(\left(
\begin{array}{cc}
 a & b  \\
 c & d
\end{array}
\right)\right)=\left(
\begin{array}{cccc}
 a & b & 0 & 0 \\
 c & d & 0 & 0 \\
 0 & 0 & a & b \\
 0 & 0 & c & d \\
\end{array}
\right)
\] with
\[ S=\left(
       \begin{array}{cccc}
         \alpha_1 & \alpha_2 & 0 & 0 \\
         0 & 0 & \beta_1 & \beta_2 \\
       \end{array}
     \right) \ \
     \mbox{ and } \ \
     T=\left(
         \begin{array}{cc}
           \xi_1 & 0 \\
           \xi_2 & 0 \\
           0 & \gamma_1 \\
           0 & \gamma_2 \\
         \end{array}
       \right).
 \]
\noindent We remark that this is the principle dilation for the following maps $d$ and $\phi$ on $\M_2$:
\[ d\left(\left(
            \begin{array}{cc}
              a & b \\
              c & d \\
            \end{array}
          \right)\right)=\left(
            \begin{array}{cc}
              a & 0 \\
              0 & d \\
            \end{array}
          \right) \ \ \mbox{ and } \ \
          \phi\left(\left(
            \begin{array}{cc}
              a & b \\
              c & d \\
            \end{array}
          \right)\right)=\left(
            \begin{array}{cc}
               a & 0 \\
              0 & a \\
            \end{array}
          \right).\]
\end{example}

\end{document}